\documentclass{amsart}
\usepackage{amsfonts}
\usepackage{amsmath,amssymb}
\usepackage{amsthm}
\usepackage{amscd}
\usepackage{graphics}
\usepackage{graphicx}

\usepackage[pagewise]{lineno}
\theoremstyle{remark}{
\newtheorem{Def}{{\rm Definition}}

}
\theoremstyle{plain}{

\newtheorem{Prop}{Proposition}
\newtheorem{Thm}{Theorem}
\newtheorem*{MainThm}{Main Theorem}

}

\begin{document}
\title[New families of fold maps on $7$-dimensional simply-connected manifolds]{$7$-dimensional closed simply-connected and spin manifolds having 2nd integral cohomology classes whose squares are not divisible by $2$ and stable fold maps on them}
\author{Naoki Kitazawa}
\address{Institute of Mathematics for Industry, Kyushu University, 744 Motooka, Nishi-ku Fukuoka 819-0395, Japan\\
 TEL (Office): +81-92-802-4402 \\
 FAX (Office): +81-92-802-4405 \\
}
\email{n-kitazawa@imi.kyushu-u.ac.jp}
\urladdr{https://naokikitazawa.github.io/NaokiKitazawa.html}
\maketitle
\begin{abstract}
This article presents families of $7$-dimensional closed and simply-connected manifolds
and {\it fold} maps on them such that squares of 2nd integral cohomology classes may not be divisible by $2$.

{\it Fold} maps are higher dimensional versions of Morse functions. The author has launched and been challenging the following new area: geometric and constructive studies of higher dimensional, closed and simply-connected manifolds. They are central objects in classical algebraic topology and differential topology. They were classified via algebraic and abstract objects in the last century and their understanding has been studied via concrete algebraic topological theory such as concrete bordism theory since the 2000s by Crowley, Kreck and Wang for example.  
 
Fold maps are fundamental objects in the new area and the author has obtained families of these manifolds and fold maps on the manifolds. The present paper presents a related new explicit result.


\end{abstract}


\maketitle
\section{Introduction, terminologies and notation.}
\label{sec:1}
\subsection{Smooth maps and fold maps.}
{\it Fold} maps are higher dimensional versions of Morse functions and fundamental tools in the present paper and in an area which we can regard as a higher dimensional version of the theory of Morse functions and its applications to algebraic topology and differential topology.

A {\it singular} point of a differentiable map is a point in the domain at which the dimension of the image of the differential is smaller than both the dimensions of the domain and the target. The set of all the singular points the {\it singular set} of the map. The image of the singular set is the {\it singular value set} of the map. The {\it regular value set} of the map is the complementary set of the singular value set of the map. A {\it singular {\rm (}regular{\rm )} value} means a point in the singular (resp. regular) value set.

Throughout the present paper, manifolds, maps between them, (boundary) connected sums of manifolds, and other fundamental notions are considered in the smooth category (in the class $C^{\infty}$) unless otherwise stated. 
For a smooth map $c$, $S(c)$ denotes the singular set of $c$.

\begin{Def}
\label{def:1}
Let $m \geq n \geq 1$ be integers.
A smooth map from an $m$-dimensional smooth manifold with no boundary into an $n$-dimensional smooth manifold with no boundary is said to be a {\it fold} map if at each singular point $p$, it has the form
$$(x_1, \cdots, x_m) \mapsto (x_1,\cdots,x_{n-1},\sum_{k=n}^{m}{x_k}^2)$$
for suitable coordinates and a suitable integer satisfying $0 \leq i(p) \leq \frac{m-n+1}{2}$.
\end{Def}

\begin{Prop}
\label{prop:1}
For a fold map $f$ in Definition \ref{def:1}, $S(f)$ is an {\rm (}$n-1${\rm )}-dimensional closed and smooth submanifold with no boundary, $f {\mid}_{S(f)}$ is an immersion and $i(p)$ is unique for any $p \in S(f)$.
\end{Prop}
We call $i(p)$ the {\it index} of $p$ there.
\begin{Def}
\label{def:2}
A fold map is said to be {\it special generic} if the index $i(p)$ is always $0$ for any singular point $p$.
\end{Def}
\subsection{Some classes of 7-dimensional closed and simply-connected manifolds.}
${\mathbb{R}}^k$ denotes the $k$-dimensional Euclidean space.
For a point $x \in {\mathbb{R}}^n$, $||x||$ denotes the distance between $x$ and the origin $0$ where the underlying metric is the standard Euclidean metric: it also denotes the value of the Euclidean norm at the vector $x \in {\mathbb{R}}^n$.

$S^k:=\{x \in {\mathbb{R}}^{k+1}\mid ||x||=1.\}$ denotes the $k$-dimensional unit sphere and $D^k:=\{x \in {\mathbb{R}}^{k}\mid ||x||=1.\}$ denotes the $k$-dimensional unit disk. ${\mathbb{C}P}^k$ denotes the $k$-dimensional complex projective space, which is a $k$-dimensional closed and simply-connected complex manifold.

A {\it homotopy} sphere means a smooth manifold homeomorphic to a unit disk. It is said to be {\it exotic} if it is not diffeomorphic to any unit sphere. If it is diffeomorphic to a $k$-dimensional unit sphere, then it is a $k$-dimensional {\it standard} sphere.

$7$-dimensional closed and simply-connected manifolds are important objects in the theory of classical algebraic topology and differential topology (of higher dimensional closed and simply-connected manifolds).
\begin{itemize}
\item There exist exactly $28$ types of $7$-dimensional oriented homotopy spheres. See \cite{milnor}, which is a pioneering work, and see also \cite{eellskuiper} for example.
\item (\cite{crowleyescher}, \cite{crowleynordstrom}, and so on.)
$7$-dimensional closed and $2$-connected manifolds are classified via concrete algebraic topological theory.
\item (\cite{kreck}.) The previous classification is extended to one for the class of $7$-dimensional closed and simply-connected manifolds whose 2nd integral homology groups are free.
\item(\cite{wang}.) There exists a one-to-one correspondence between the topologies of $7$-dimensional, closed, simply-connected and spin manifolds whose integral cohomology rings are isomorphic to that of ${\mathbb{C}P}^2 \times S^3$ and 2nd integral cohomology classes which are divisible by $4$. A closed and oriented manifold $X$ having such a topology is represented as a connected sum of $X$ and a $7$-dimensional oriented homotopy sphere $\Sigma$. Furthermore, between two of these manifolds, there exists an orientation-preserving diffeomorphism if and only if between the oriented homotopy spheres appearing in the definition, there exists an orientation-preserving diffeomorphism.  
\end{itemize}

\subsection{Fold maps on $7$-dimensional closed and simply-connected manifolds and a main theorem.}

The class of special generic maps contains Morse functions with exactly two singular points on homotopy spheres, playing important roles in so-called Reeb's theorem and canonical projections of unit spheres, for example. Related to this and more general fundamental theory of Morse functions and decompositions of the manifolds into {\it handles} associated with the functions, see \cite{milnor2} for example.

Proving that canonical projections of unit spheres are fold maps and special generic is, exercises on fundamental theory of differentiable manifolds and maps and Morse functions.

\begin{Thm}[\cite{calabi}, \cite{saeki}, \cite{saeki2}, \cite{wrazidlo} and so on.]
\label{thm:1}
$7$-dimensional {\it exotic} homotopy spheres admit no special generic maps into ${\mathbb{R}}^n$ for $n=4,5,6$. Oriented exotic homotopy spheres of 14 types admit no special generic maps for $n=3$.
\end{Thm}

A bundle whose fiber is a (smooth) manifold is assumed to be {\it smooth} unless otherwise stated. A {\it smooth} bundle is a bundle whose structure group is a subgroup of the diffeomorphism group: the {\it diffeomorphism group} of a (smooth) manifold is the group of all diffeomorphisms (of course the diffeomorphisms are assumed to be smooth).

\begin{Thm}[\cite{kitazawa2}]
\label{thm:2}
Every $7$-dimensional homotopy sphere admits a fold map into ${\mathbb{R}}^4$ such that $f {\mid}_{S(f)}$ is an embedding and that $f(S(f))=\{x \mid ||x||=1,2,3.\}$. Furthermore, we have the following three.
\begin{enumerate}
\item We can obtain this map $f$ so that for any connected component $C \subset f(S(f))$, there exists a small closed tubular neighborhood $N(C)$ so that the composition of $f {\mid}_{f^{-1}(N(C))}:f^{-1}(N(C)) \rightarrow N(C)$ with the canonical projection to $C$ gives a trivial smooth bundle.
\item In the previous statement, we can replace $f(S(f))=\{x \mid ||x||=1,2,3.\}$ by $f(S(f))=\{x \mid  ||x||=1.\}$ if and only if the homotopy sphere is diffeomorphic to the unit sphere. 
\item As the previous statement, we can replace $f(S(f))=\{x \mid ||x||=1,2,3.\}$ by $f(S(f))=\{x \mid ||x||=1,2.\}$ if and only if the homotopy sphere is represented as the total space of a smooth bundle over $S^4$ whose fiber is diffeomorphic to $S^3$. There exist exactly $16$ types of $7$-dimensional oriented homotopy spheres satisfying this property by virtue of \cite{eellskuiper}. 
\end{enumerate}
\end{Thm}
\begin{Thm}[\cite{kitazawa8}]
\label{thm:3}
Every $7$-dimensional manifold of \cite{wang}, presented in the previous subsection, admits a fold map into ${\mathbb{R}}^4$ such that $f {\mid}_{S(f)}$ is an embedding and that $f(S(f))=\{x \mid 1 \leq ||x|| \in \mathbb{N} \leq l\}$ for a suitable integer $3 \leq l \leq 5$. 
\end{Thm}
Hereafter, $\mathbb{N}$ denotes the set of all positive integers. 
\begin{MainThm}
Let $l_0>0$ be an integer. Let $\{l_j\}_{j=1}^{l_0}$ and $\{k_j\}_{j=1}^{l_0}$ be sequences of integers of length $l_0$. Let ${l_0}^{\prime} \leq l_0$ be a positive integer. 
Let ${\Pi}_{l_0,{l_0}^{\prime}}$ be a surjection from the set $\{p \in \mathbb{N} \mid 1 \leq p \leq l_0.\}$ onto the set $\{p \in \mathbb{N} \mid 1 \leq p \leq {l_0}^{\prime}.\}$. 
Then there exist a $7$-dimensional closed, oriented, simply-connected and spin manifold $M$ and a {\rm stable} fold map $f:M \rightarrow {\mathbb{R}}^4$ such that the following four properties hold.
\begin{enumerate}
\item $H_2(M;\mathbb{Z}) \cong {\mathbb{Z}}^{l_0}$ and $H_3(M;\mathbb{Z}) \cong {\mathbb{Z}}^{{l_0}^{\prime}}$.
\item For suitable bases $\{a_j\}_{j=1}^{l_0} \subset H^2(M;\mathbb{Z}) \cong {\mathbb{Z}}^{l_0}$ and $\{b_j\}_{j=1}^{{l_0}^{\prime}} \subset H^4(M;\mathbb{Z}) \cong {\mathbb{Z}}^{{l_0}^{\prime}}$, the following two properties hold.
\begin{enumerate}
\item The cup product of $a_{j_1}$ and $a_{j_2}$ always vanishes for distinct $j_1$ and $j_2$.
\item The square of $a_j$ is $l_j b_{{\Pi}_{l_0,{l_0}^{\prime}}(j)}$.
\end{enumerate}
\item The 1st Pontryagin class of $M$ is ${\Sigma}_{j=1}^{l_0} k_j b_{{\Pi}_{l_0,{l_0}^{\prime}}(j)}$.
\item Connected components of the singular set of $f$ are diffeomorphic to the $3$-dimensional unit sphere. The restriction to each connected component of the singular set of $f$ is an embedding and the preimage of a singular value contains at most $2$ singular points. In addition, connected components of preimages of regular values are always diffeomorphic to $S^3$ or $S^2 \times S^1$. 
\end{enumerate}
\end{MainThm}
Different from results on explicit fold maps into ${\mathbb{R}}^4$ on $7$-dimensional, closed, simply-connected and spin manifolds in \cite{kitazawa6}, \cite{kitazawa7} and \cite{kitazawa9}, squares of 2nd integral cohomology classes in Theorem \ref{thm:3} and Main Theorem may not be divisible by $2$. 
The definition of a {\it stable} fold map and a proof of Main Theorem are presented in the next section. 

\section{A proof of main theorem.}
\label{sec:2}
For a smooth manifold $X$, $T_pX$ denotes the tangent vector space at $p \in X$. For a smooth map $c:X \rightarrow Y$, ${dc}_p:T_pX \rightarrow T_{c(p)} Y$ denotes the differential at $p$.
\begin{Def}
\label{def:3}
A fold map $c:X \rightarrow Y$ on a closed manifold is said to be {\it stable} if for any $y \in c(S(c))$, the following two conditions are satisfied.
\begin{enumerate}
\item $c^{-1}(y)$ is a discrete set $\{p_j\}_{j=1}^l$ consisting of exactly $l>0$ points.
\item $\dim Y=\dim {\bigcap}_{j=1}^l {dc}_{p_j}(T_{p_j}X)+{\Sigma}_{j=1}^l (\dim Y-{\rm rank}\ {dc}_{p_j})$.
\end{enumerate} 
\end{Def}
Note that $\dim Y-{\rm rank}\ {dc}_{p_j}=1$ there.
For example, a fold map $f$ such that $f {\mid}_{S(f)}$ is an embedding is stable. 
For this notion, for example, see also \cite{golubitskyguillemin}, which mainly explains fundamental theory and classical important theory on singularities of differentiable maps.
\begin{Def}[\cite{kitazawa}--\cite{kitazawa5}.]
\label{def:4}
Let $m \geq n \geq 2$ be integers.
A stable fold map $f:M \rightarrow {\mathbb{R}}^n$ on an $m$-dimensional closed manifold $M$ into ${\mathbb{R}}^n$ is said to be {\it round} if $f {\mid}_{S(f)}$ is an embedding and for a suitable integer $l>0$ and a suitable diffeomorphism $\phi$ on ${\mathbb{R}}^n$, $(\phi \circ f)(S(f))=\{x \mid 1 \leq ||x|| \in \mathbb{N} \leq l.\}$.
\end{Def}

This class contains canonical projections of unit spheres into the Euclidean space whose dimension is greater than $1$, stable fold maps in Theorems \ref{thm:2} and \ref{thm:3}, and so on. We can define this notion for $n=1$ and this class contains Morse functions with exactly two singular points on homotopy spheres (\cite{kitazawa5}).

A {\it linea}r bundle is a smooth bundle whose fiber is diffeomorphic to a unit sphere or a unit disk and whose structure group acts linearly on the fiber. For linear bundles, characteristic classes such as {\it Stiefel-Whitney classes}, {\it Euler classes} (for {\it oriented} linear bundles), {\it Pontryagin classes}, and so on, are defined as cohomology classes of the base spaces. 
We can define these notions for smooth manifolds as those of the tangent bundles.
We do not review oriented linear bundles, these characteristic classes, or so on, precisely. See \cite{milnorstasheff} for example. See also \cite{little} for Pontryagin classes of complex projective spaces and manifolds which are homotopy equivalent to them.

For a closed manifold, a homology class is said to be {\it represented by a closed {\rm (}and oriented{\rm )} submanifold with no boundary} if it is realized as the value of the homomorphism induced by the canonical inclusion map of the submanifold at the {\it fundamental class}. The fundamental class of a closed (oriented) manifold is a generator of the top homology group of the manifold (compatible with the orientation) for a coefficient commutative group: if the group is isomorphic to $\mathbb{Z}/2\mathbb{Z}$, then we do not need to orient the manifold.

For a compact manifold $X$, let $c$ be an element of the $i$-th homology group $H_i(X;\mathbb{Z})$ such that we cannot represent as $c=rc^{\prime}$ for any $(r,c^{\prime}) \in (\mathbb{Z}-\{0,1,-1\}) \times H_i(X;\mathbb{Z})$ and that $rc \neq 0$ for any $r \in \mathbb{Z}-\{0\}$. We can define $c^{\ast} \in H^{i}(X;\mathbb{Z})$ such that $c^{\ast}(c)=1$ and that $c^{\ast}(A)=0$ for any subgroup $A \subset H_i(X;\mathbb{Z})$ giving a representation of $H_i(X;\mathbb{Z})$ as the internal direct sum of the subgroup generated by $c$ and $A$. We can define this in a unique way and call this the {\it dual} of $c$.

The following proposition extends (some of) Theorem \ref{thm:3}.

\begin{Prop}
\label{prop:2}
Let $l$ be an integer. There exists a family $\{f_{l,k}:M_{l,k} \rightarrow {\mathbb{R}}^4\}_{k \in\mathbb{Z}}$ of round fold maps on $7$-dimensional oriented, closed, simply-connected and spin manifolds satisfying the following four properties.
\begin{enumerate}
\item
\label{prop:2.1}
$H_j(M_{l,k};\mathbb{Z}) \cong \mathbb{Z}$ for $j=0, 2, 3, 4, 5, 7$.
\item
\label{prop:2.2}
For a suitable generator of $H^2(M_{l,k};\mathbb{Z}) \cong \mathbb{Z}$, the square is $l$ times a suitable generator of $H^4(M_{l,k};\mathbb{Z})$.
\item
\label{prop:2.3}
The 1st Pontryagin class of $M_{l,k}$ is $4k$ times the generator of $H^4(M_{l,k};\mathbb{Z})$ defined just before.
\item
\label{prop:2.4}
The singular value set of each round fold map consists of exactly three connected components and $\{x \in {\mathbb{R}}^4 \mid ||x||=1,2,3.\}$. Furthermore, the preimages of a point in $\{x \in {\mathbb{R}}^4 \mid ||x||<1.\}$, one in $\{x \in {\mathbb{R}}^4 \mid 1<||x||<2.\}$, and one in $\{x \in {\mathbb{R}}^4 \mid 2<||x||<3.\}$, are diffeomorphic to $S^2 \times S^1 \sqcup S^3$, $S^2 \times S^1$ and $S^3$, respectively.
\end{enumerate} 
\end{Prop}
\begin{proof}
We can prove this in a similar manner to that in the proof of Theorem \ref{thm:2} in \cite{kitazawa8}. 
We prove this including this original case.  

There exists a linear bundle ${M^{\prime}}_l$ over $S^4$ whose fiber is diffeomorphic to $S^2$ satisfying the following two properties.
\begin{enumerate}
\item
\label{prop:2.0-1.1}
 The square of a generator of the 2nd integral cohomology group of this total space is $l$ times a generator of the $4$-th integral cohomology group.
\item
\label{prop:2.0-1.2}
The 1st Pontryagin class of the total space is $4l$ times the generator before. 
\end{enumerate}
For this manifold, see articles and webpages on $6$-dimensional closed and simply-connected manifolds such as \cite{jupp}, \cite{wall}, \cite{zhubr} and \cite{zhubr2} for example.
If $l=0$, then the total space is diffeomorphic to $S^2 \times S^4$ and if $l=1$, then it is diffeomorphic to the $3$-dimensional complex projective space ${\mathbb{C}P}^3$. We have a smooth bundle ${M^{\prime}}_l \times S^1$ over $S^4$ whose fiber is diffeomorphic to $S^2 \times S^1$ by the composition of the projection of a trivial bundle over ${M^{\prime}}_l$ with the projection of the linear bundle before. We remove the interior of a smoothly embedded copy of a $4$-dimensional unit disk and the preimage for the resulting projection. We can exchange this to a stable fold map such that the restriction to the singular set is an embedding, that the singular set is diffeomorphic to $S^3$, and that the preimages of regular values are $S^3$ and $S^2 \times S^1$, respectively, for the two connected components of the regular value set. Furthermore, we can do this so that the following three hold. Rigorous understandings are left to readers and similar expositions are in \cite{kitazawa8} for $l=1$. 
\begin{enumerate}
\item
\label{prop:2.0-2.1}
If an arbitrary integer $k$ is given, then the domain is a $7$-dimensional oriented, closed, simply-connected and spin manifold $M_{l,k}$ and this satisfies the following properties.
\begin{enumerate}
\item
\label{prop:2.0-2.1.1}
$H_j(M_{l,k};\mathbb{Z}) \cong \mathbb{Z}$ for $j=0, 2, 3, 4, 5, 7$.
\item
\label{prop:2.0-2.1.2}
We take a 2nd integral homology class $c$ represented by $S^2 \times \{\ast\} \subset S^2 \times S^1$ in the preimage of a regular value before. For its dual, which is a cohomology class in $H^2(M_{l,k};\mathbb{Z}) \cong \mathbb{Z}$, the square is $l$ times a generator of $H^4(M_{l,k};\mathbb{Z})$, which is the dual of an integral homology class represented by a suitable closed submanifold with no boundary and also the Poincar\'e dual to a homology class represented by the preimages of regular values, which are diffeomorphic to $S^3$ or $S^2 \times S^1$.
\item
\label{prop:2.0-2.1.3}
The 1st Pontryagin class of $M_{l,k}$ is $4k$ times the generator of $H^4(M_{l,k};\mathbb{Z})$ defined just before.
\end{enumerate}
\item
\label{prop:2.0-2.2}
For the singular value set $C$, there exists a small closed tubular neighborhood $N(C)$ and the composition of the restriction to the preimage of $N(C)$ with a canonical projection to $C$ gives a trivial smooth bundle whose fiber is diffeomorphic to a manifold obtained by removing the interior of a copy of the $4$-dimensional unit disc smoothly embedded in a manifold diffeomorphic to $S^2 \times {\rm Int}\ D^2 \subset S^2 \times D^2$.
\item
\label{prop:2.0-2.3}
The index of each singular point is $2$.   
\end{enumerate}
If $l=0$, then the suitable closed submanifold in the first property here can be taken as one diffeomorphic to the $4$-dimensional unit sphere.
If $l=1$, then the suitable closed submanifold in the first property here can be taken as one diffeomorphic to the complex projective plane ${\mathbb{C}P}^2$.

We construct a desired round fold map from this surjection. We can construct a trivial bundle over the subset $\{x \in {\mathbb{R}}^4 \mid ||x|| \leq \frac{1}{2}.\} \subset {\mathbb{R}}^4$
whose fiber is diffeomorphic to $S^3 \sqcup (S^2 \times S^1)$ and whose total space is the preimage of the complementary set of ${\rm Int}\ N(C) \subset S^4$. 
On the preimage of $N(C)$ for the surjection, we have the product map of a suitable Morse function on a $4$-dimensional compact manifold diffeomorphic to the fiber of the trivial bundle over $C$ before and the identity map on $C$ and glue this and the previous projection in a suitable way on the boundaries. Thus we have a desired round fold map into ${\mathbb{R}}^4$. If $l=1$, then we have a (partial) proof of Theorem \ref{thm:3}.
\end{proof} 
Via fundamental arguments on deformations of Morse functions and stable fold maps, we immediately have the following proposition.
\begin{Prop}
\label{prop:3}
Let $l$ be an integer. There exists a family $\{F_{l,k}:M_{l,k} \times [0,1] \rightarrow {\mathbb{R}}^4\}_{k \in\mathbb{N}}$ of smooth homotopies such that for any $t \in [0,1]$ the maps $F_{l,k,t}$ mapping $x$ to $F_{l,k}(x,t)$ are fold maps on the $7$-dimensional oriented, closed, simply-connected and spin manifolds satisfying the following four properties.
\begin{enumerate}
\item
\label{prop:3.1}
$F_{l,k,0}=f_{l,k}$.
\item
\label{prop:3.2}
The singular set $S(F_{l,k,t})$ consists of exactly three copies of the $3$-dimensional unit sphere and on each connected component it is an embedding. Furthermore, $F_{l,k,t}(S(F_{l,k,t}))=\{x \in {\mathbb{R}}^4 \mid ||x||=2,3.\} \bigcup \{x:=(x_1,x_2,x_3,x_4) \in {\mathbb{R}}^4 \mid ||x-(\frac{3}{2}t,0,0,0)||=1.\}$.
\item
\label{prop:3.3}
For each map $F_{l,k,t}$, the preimages of a point in $\{x \in {\mathbb{R}}^4 \mid ||x||<2, ||x-(\frac{3}{2}t,0,0,0)||>1.\}$, one in $\{x \in {\mathbb{R}}^4 \mid 2<||x||<3,||x-(\frac{3}{2}t,0,0,0)||>1.\}$, one in $\{x \in {\mathbb{R}}^4 \mid ||x||<2, ||x-(\frac{3}{2}t,0,0,0)||<1.\}$, and one in $\{x \in {\mathbb{R}}^4 \mid 2<||x||<3, ||x-(\frac{3}{2}t,0,0,0)||<1.\}$, are diffeomorphic to $S^2 \times S^1$, $S^3$, $(S^2 \times S^1) \sqcup S^3$ and $S^3 \sqcup S^3$, respectively.  
\item
\label{prop:3.4}
Furthermore, in the previous situation, for the preimages diffeomorphic to $(S^2 \times S^1) \sqcup S^3$ and $S^3 \sqcup S^3$, the latter $S^3$'s for these two manifolds are isotopic to $S^3$ in $(S^2 \times S^1) \sqcup S^3$ of a preimage for the map $F_{l,k,0}=f_{l,k}$ where suitable identifications of preimages are considered.
\end{enumerate} 
\end{Prop}
For this, see also FIGURE \ref{fig:1}.
\begin{figure}
\includegraphics[width=40mm]{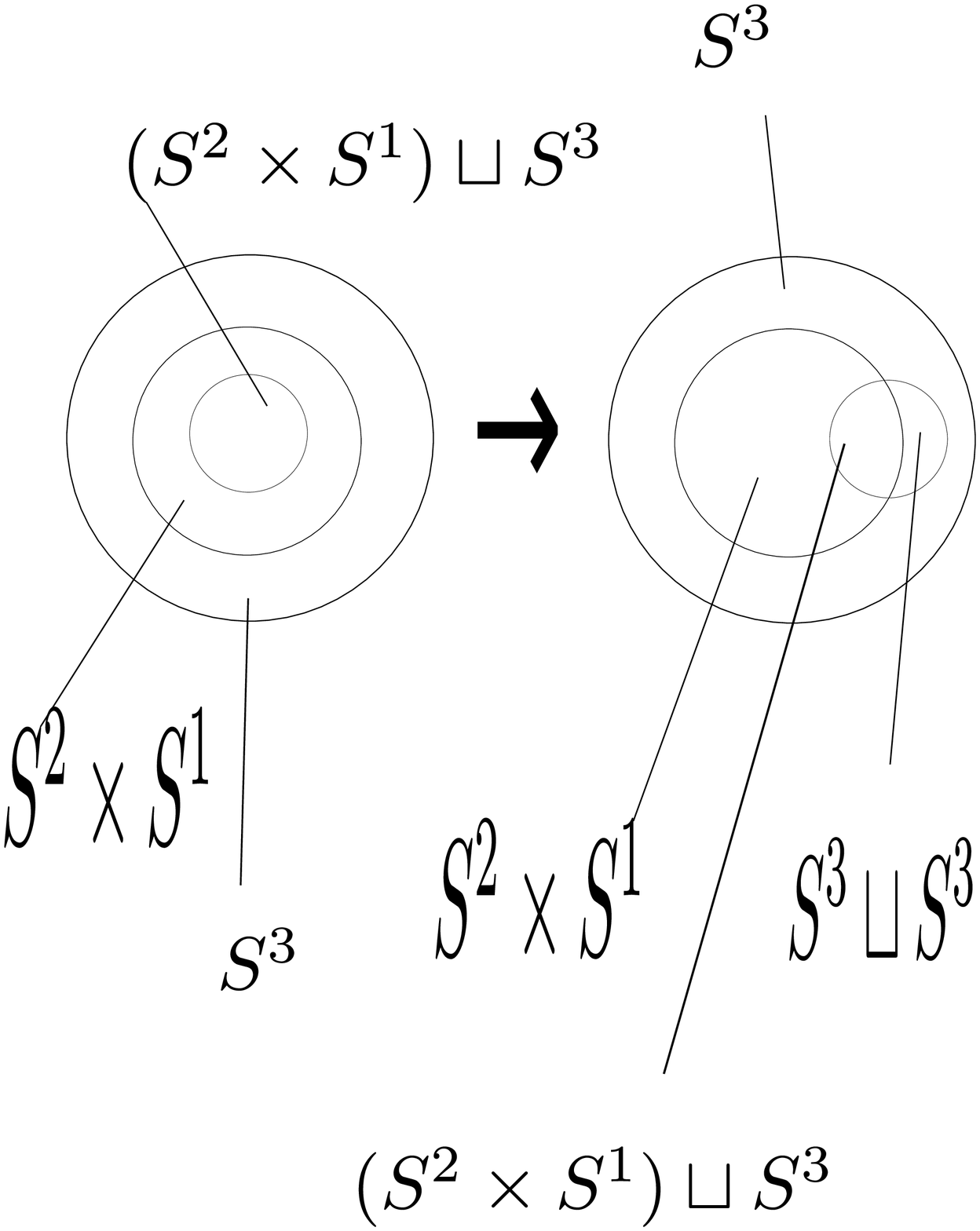}
\caption{The singular value sets of $F_{l,k,0}=f_{l,k}$ and $F_{l,k,1}$ in Proposition \ref{prop:3}: descriptions of the manifolds represent preimages.}
\label{fig:1}
\end{figure}
For a compact manifold, a closed submanifold is said to be {\it proper} if the boundary of the closed submanifold is in the boundary of the compact manifold and the interior of the submanifold is in the interior of the manifold.
We prove the main theorem.
\begin{proof}[A proof of Main Theorem]
We can easily see that $F_{l,k,1}$ is a stable fold map. $P$ denotes the subset $\{x:=(x_1,x_2,x_3,x_4) \in {\mathbb{R}}^4 \mid x_1 \geq 2.25.\}$ of ${\mathbb{R}}^4$.
We take two copies of the restriction of this stable fold map $F_{l,k,1} {\mid}_{{F_{l,k,1}}^{-1}({\mathbb{R}}^4-{\rm Int}\ P)}$. This is a smooth map on the manifold obtained by removing the tubular neighborhood of a submanifold diffeomorphic to $S^3$ in $M_{l,k}$. More precisely, this submanifold can be taken as the preimage of a regular value of a round fold map in Proposition \ref{prop:2} by (the property (\ref{prop:3.4}) in) Proposition \ref{prop:3}. By gluing these copies, we have a stable fold map on a new $7$-dimensional oriented, closed, simply-connected and spin manifold $M$. We consider the following Mayer-Vietoris sequence
$$\begin{CD}
@> >> H_j(\partial {F_{l,k,1}}^{-1}({\mathbb{R}}^4-{\rm Int}\ P);\mathbb{Z}) \\
@> >> H_j( {F_{l,k,1}}^{-1}({\mathbb{R}}^4-{\rm Int}\ P);\mathbb{Z}) \oplus H_j({F_{l,k,1}}^{-1}({\mathbb{R}}^4-{\rm Int}\ P);\mathbb{Z}) \\
@> >> H_j(M;\mathbb{Z}) \\
@> >> 
\end{CD}$$
and we can see that $H_j(\partial {F_{l,k,1}}^{-1}({\mathbb{R}}^4-{\rm Int}\ P);\mathbb{Z})$ is zero for $j \neq 0,3,6$, isomorphic to $\mathbb{Z}$ for $j=0,6$, and isomorphic to $\mathbb{Z} \oplus \mathbb{Z}$ for $j=3$. Furthermore, $(\partial {F_{l,k,1}}^{-1}({\mathbb{R}}^4-{\rm Int}\ P)$ is the manifold of the domain of a round fold map whose singular set consists of two connected components and for each point in each connected component of the regular value set of this round fold map, the preimage is empty, diffeomorphic to $S^3$ and $S^3 \times S^3$, respectively. The $6$-dimensional manifold is simply-connected and diffeomorphic to $S^3 \times S^3$. For these arguments, consult Theorem 4 and Example 6 of \cite{kitazawa3} and as closely related papers \cite{kitazawa} and \cite{kitazawa2}. We can easily see that ${F_{l,k,1}}^{-1}({\mathbb{R}}^4-{\rm Int}\ P)$ is diffeomorphic to $S^3 \times D^4$ by the structures of the maps.

The kernel of the homomorphism from $H_3(\partial {F_{l,k,1}}^{-1}({\mathbb{R}}^4-{\rm Int}\ P);\mathbb{Z})$ into $H_3( {F_{l,k,1}}^{-1}({\mathbb{R}}^4-{\rm Int}\ P);\mathbb{Z}) \oplus H_3({F_{l,k,1}}^{-1}({\mathbb{R}}^4-{\rm Int}\ P);\mathbb{Z})$ is isomorphic to $\mathbb{Z}$. We see this. We can take a basis of $H_3(\partial {F_{l,k,1}}^{-1}({\mathbb{R}}^4-{\rm Int}\ P);\mathbb{Z})$ so that the two elements satisfy the following by the structures of the maps and the manifolds.
\begin{enumerate}
\item One of the two elements is represented by the preimage of a regular value of the original round fold map in Proposition \ref{prop:2}.
\item The remaining element is represented by the boundary of a $4$-dimensional proper compact submanifold in ${F_{l,k,1}}^{-1}({\mathbb{R}}^4-{\rm Int}\ P)$. Furthermore, the $4$-dimensional compact submanifold can be taken as a manifold obtained by removing the interior of a copy of the $4$-dimensional unit disc smoothly embedded in the $4$-dimensional closed submanifold with no boundary in (\ref{prop:2.0-2.1.2}) in the proof of Proposition \ref{prop:2}. In addition, we can glue $4$-dimensional closed submanifolds on the boundaries and we have a new $4$-dimensional closed submanifold $S$ with no boundary in the resulting manifold $M$.
\end{enumerate}
On the submodule generated by the first element, the homomorphism is a monomorphism. On the submodule generated by the second element, it is zero. This completes the proof on this fact on the  kernel. We have $H_j( {F_{l,k,1}}^{-1}({\mathbb{R}}^4-{\rm Int}\ P);\mathbb{Z}) \oplus H_j({F_{l,k,1}}^{-1}({\mathbb{R}}^4-{\rm Int}\ P);\mathbb{Z}) \cong H_j(M;\mathbb{Z})$ for $j=2,5$ easily. Furthermore, by the argument on the kernel and the structure of the homomorphism, the group $H_3(M;\mathbb{Z})$ is free and its rank is\\
${\rm rank} \quad (H_3( {F_{l,k,1}}^{-1}({\mathbb{R}}^4-{\rm Int}\ P);\mathbb{Z}) \oplus H_3({F_{l,k,1}}^{-1}({\mathbb{R}}^4-{\rm Int}\ P);\mathbb{Z}))-1$.
We can know the integral homology group and the integral cohomology ring of $M$ by virtue of Poincar\'e duality theorem and the topological structures of the manifolds and the maps. Furthermore, for example, $H_4(M;\mathbb{Z})$ is generated by a homology class represented by the submanifold $S$.

By virtue of the arguments (together with some additional properties on topological structures), we have a desired stable fold map on a desired manifold in the case $(l_0,{l_0}^{\prime})=(2,1)$ with $(l_1,l_2,k_1,k_2)=(l,l,k,k)$. For example, we can take $a_1$ and $a_2$ as the duals of natural homology classes and $b_1$ as the dual of a class represented by $S$ before.

By considering a connected sum of the original manifolds instead, we have a desired result in the case $(l_0,{l_0}^{\prime})=(2,2)$ with $(l_1,l_2,k_1,k_2)=(l,l,k,k)$. In this case we define $P^{\prime}:=\{x:=(x_1,x_2,x_3,x_4) \in {\mathbb{R}}^4 \mid x_1 \geq 2.75.\}$ of ${\mathbb{R}}^4$ instead of $P$ to obtain a desired map and a manifold. We can also consider original round fold maps in Proposition \ref{prop:2} or \ref{prop:3} and take $P^{\prime}$ instead as before to obtain a desired map and a manifold. We have a desired stable fold map.

We can easily see that we can generalize the proofs to general cases. In fact it is sufficient to consider suitable iterations of the presented fundamental operations (in suitably and naturally generalized ways) for given stable maps into ${\mathbb{R}}^4$. 

This completes the proof. 
\end{proof}
\section{Acknowledgement.}
\label{sec:3}
\thanks{This work was partially supported by "The Sasakawa Scientific Research Grant" (2020-2002 : https://www.jss.or.jp/ikusei/sasakawa/). The author is a member of JSPS KAKENHI Grant Number JP17H06128 "Innovative research of geometric topology and singularities of differentiable mappings" (Principal Investigator: Osamu Saeki). This work is also supported by the project. We declare that data supporting our present work is all in the present paper.}

\end{document}